\documentclass[11pt, DIV=15, a4paper]{article}

\usepackage{geometry}
\usepackage{etex} 
\usepackage[utf8]{inputenc}
\usepackage{fancyhdr}
\usepackage{bm}
\usepackage{graphicx, pgf, pst-all}
\usepackage{tikz}
\usetikzlibrary{arrows, intersections, positioning ,calc, plothandlers, shapes, decorations}
\usepackage{array,multirow}
\usepackage{caption}
\usepackage{subcaption}
\usepackage{float}
\floatstyle{plaintop} 
\restylefloat{table} 
\usepackage[justification=centering]{caption} 
\usepackage{hyperref}
\usepackage[bottom]{footmisc}
\DeclareCaptionFont{xipt}{\fontsize{8}{10}\mdseries}
\usepackage[font=xipt,labelfont=bf]{caption}
\usepackage{here}
\renewcommand\bf\bfseries
\newcommand{\n}{\negthinspace}
\newcommand{\nn}{\negthinspace\negthinspace\negthinspace\negthinspace\negthinspace\negthinspace\negthinspace\negthinspace}

\usepackage{amssymb,amsthm,amsmath}
\usepackage{cancel}
\usepackage{enumerate} 
\usepackage[squaren]{SIunits}
\usepackage[autostyle=true]{csquotes}
\usepackage{mathtools}
\usepackage{MnSymbol}
\usepackage{pgfplots}
\usetikzlibrary{calc}

\newcommand{\R}{\operatorname{\mathbb R}}

\newcommand{\norm}[1]{\ensuremath{\left|\left|#1\right|\right|}}

\newcommand{\half}{\frac{1}{2}}
\newcommand{\limint}{\int\limits}
\newcommand{\dx}{\: \mathrm{d}x}

\newtheorem{theorem}{Theorem}
\newtheorem{definition}{Definition}
\newtheorem{lemma}[theorem]{Lemma}

\newtheorem{remark}{Remark}

\DeclareMathOperator{\Div}{\mathbf{div}}

\newcommand*{\defeq}{\mathrel{\vcenter{\baselineskip0.5ex \lineskiplimit0pt
                     \hbox{\scriptsize.}\hbox{\scriptsize.}}}
                     =}

\begin{document}
\title{On a saddle point problem arising from magneto-elastic coupling}
\author{Mané Harutyunyan, Bernd Simeon} 
\maketitle


\section{Introduction}\label{sec:Intro}

Coupled systems of partial differential equations arise in various applications where the interplay of different physical phenomena is of major relevance and cannot be neglected. Typical examples for such multiphysics problems are fluid-structure interfaces, thermo-elasticity, and magneto-elasticity, which is the problem class we study here. The computational treatment of such coupled problems has seen great advances over the last decade, but the underlying problem structure is often not fully understood nor taken into account when using black box simulation codes for the individual subproblems.\\ 
For this reason, a careful analysis of the properties of a coupled problem is an important task, and in this paper we aim at such a comprehensive treatment for linear magneto-elastostatic solids. Possible application fields of this material class include robotics, vibration control, hydraulics and sonar systems. Magneto-elastic materials come also into operation as variable-stiffness devices, sensors and actuators in mechanical systems or artificial muscles.\\
Basic references providing a theoretical background on magneto-elastic coupling are, among others, Brown \cite{Brown1966}, Pao \cite{Pao1978}, Truesdell and Taupin \cite{TruesdellTaupin1960} and Hutter and Van den Veen \cite{HutterVandenVeen1978}, while e.g. Engdahl (ed.) \cite{Engdahl2000} gives a thorough insight into the modeling and application of magnetostrictive materials.\\
The paper is organized as follows. In the first section, we derive a coupled magneto-elastic model for the stationary case based on linear constitutive relations for piezomagnetic materials. Section 2 is concerned with the question of existence and uniqueness of the solution to the coupled saddle point problem obtained in Section 1. After discussing important properties of the bilinear forms of the resulting model, we show that the coupled magneto-elastic bilinear form satisfies an inf-sup condition that implies both the uniqueness of the solution and the stability of the coupled saddle point problem. In Section 3 we transfer the results obtained for the continuous case to the discrete case and prove that the inf-sup condition is also fulfilled in the discrete case using standard Lagrangian finite elements, as long as the elastic and magnetic basis functions are chosen to be polynomials of the same degree. Finally, a short conclusion can be found in Section 4.

\section{The coupled problem} \label{sec:Section1}

We consider materials that are {\em magnetostrictive}, i.e., that change their shape or dimensions during the process of magnetization. The material is assumed to be homogeneous at the macroscopic level  and forms a body whose undeformed state we denote by the domain $\Omega\in\R^{3}$. We aim at examining the behavior due to the influence of external magnetic and mechanical fields. Our model is based on a linearized theory, which means that only small strains and small deviations from the initial magnetized state are considered. This implies that the change of the dimensions of the magnetostrictive body does not affect the magnetic field outside the material and the magnetoelastic coupling only occurs in the material itself.

The deformation of the body is characterized by the displacement field $\bm{u}:\Omega\rightarrow \R^{3}$ and the strain tensor $\bm{\epsilon}:\Omega\rightarrow \R^{3\times 3},
 \bm{\epsilon}(\bm{u}) = (\nabla \bm{u} + \nabla \bm{u}^T)/2$, while the magnetic influence is described by the magnetic field $\bm{H}:\Omega\rightarrow \R^{3}$. Furthermore, $\bm{\sigma}:\Omega \rightarrow \R^{3\times 3}$ and $\bm{B}:\Omega \rightarrow \R^{3}$ describe the stress tensor and the magnetic flux density.
Although magnetostrictive materials show a non-linear behavior in general, 
they can be reasonably characterized
in case of small strains and magnetizations by means of the linearized constitutive equations of piezomagnetic materials \cite{Engdahl2000,IEEE1990}, which read
\begin{equation}\label{eq:constitutives}
\bm{\sigma}(\bm{u},\Psi)=\bm{C^{H}}\cdot\bm{\epsilon}(\bm {u})+\bm{e}\cdot \nabla \Psi, \quad
\bm{B}(\bm{u},\Psi)=\bm{e^{T}}\cdot\bm{\epsilon}(\bm{u})-\bm{\mu^{\epsilon}} \cdot \nabla \Psi.
\end{equation}
Here $\bm{C}^{H}$ denotes the elastic stiffness matrix for a constant magnetic field, $\bm{e}$ the magneto-elastic coupling matrix and $\bm{\mu}^{\epsilon}$ the magnetic permeability for constant strain. Note that
$\bm{\sigma}, \bm{C}^{H}, \bm{\epsilon}, \bm{e}$ and $\bm{\mu}^{\epsilon}$ are generally tensors of order 2 and higher and have been reduced to matrices and vectors using symmetries and the Voigt notation \cite{Voigt1910}. The existence of the magnetic scalar potential $\Psi \in \R$ with $\bm{H} = -\nabla \Psi$ is justified by assuming a static magnetic field that is generated by permanent magnets rather than a current-carrying coil.

The strong and the weak form of the coupled system can be derived by using a minimum energy principle for the total energy $W_{mag}(\bm{u}, \Psi) - W_{el}(\bm{u}, \Psi)$,
which is the difference between the internal magnetic and mechanical energies \cite{ClaeyssenEtal2006} defined as
\begin{equation}\label{eq:energies}
\begin{array}{rcl} \displaystyle
	W_{mag}(\bm{u},\Psi)&=& \displaystyle \frac{1}{2}\int\limits_{\Omega}\bm{H}(\Psi)\cdot\bm{B}(\bm{u},\Psi)\: \mathrm{d}x 
	- \int\limits_{\Gamma_{mag,N}}\Psi\tilde{B}\: \mathrm{d}s,  \\
	W_{\mathit{el}}(\bm{u}, \Psi) &=&
	\displaystyle \frac{1}{2}\int\limits_{\Omega}\bm{\sigma}(\bm{u},\Psi):\bm{\epsilon}(\bm{u})\: \mathrm{d}x
	 -\int\limits_{\Gamma_{N}}\bm{u}\bm{\tau} \: \mathrm{d}s, 
\end{array}
\end{equation}
where $\R^{3} \ni \bm{\tau} = \bm{\sigma}(\bm{u},\Psi)\cdot\bm{n}$ and  $\R \ni \tilde{B}=\bm{B}(\bm{u},\Psi)\cdot\bm{n}$ define the elastic and magnetic surface forces with outer normal vector $\bm{n}\in \R^{3}$ on the Neumann boundaries $\Gamma_{N}$ and $\Gamma_{mag,N}$. Inserting the constitutive equations \eqref{eq:constitutives} into the energy expressions \eqref{eq:energies} and applying the calculus of variations along with Gauss's  divergence theorem, we can derive the strong form of the coupled problem:\\
\noindent
\textbf{Problem $(CP)$}. Find $(\bm{u},\Psi) \in (C^{2}(\Omega))^{3}\times C^{2}(\Omega)$ such that
\[
\begin{array}{rcl}
\Div(\bm{C^{H}}:\bm{\epsilon}(\bm{u})) + \half \Div(\bm{e}\cdot\nabla\Psi)  &=& \bm{0} \text{\,\, in\,\,} \Omega,\\
\half\Div(\bm{e}^{T}:\bm{\epsilon}(\bm{u}))-\Div(\bm{\mu^{\epsilon}}\cdot\nabla\Psi) &=& 0 \text{\,\, in\,\,} \Omega,\\
\end{array}
\]
with Dirichlet boundary conditions $\bm{u}=\bm{0}$ on $\Gamma_{D}$ and $\Psi= 0$ on $\Gamma_{mag,D}$\\
and Neumann boundary conditions 
\[
\left(\bm{C^{H}}:\bm{\epsilon}(\bm{u})+\half\bm{e}\cdot\nabla\Psi\right)\bm{n} = \bm{\tau} \quad \text{on\,\,} \Gamma_{N}, \quad
\left(\half\bm{e}^{T}:\bm{\epsilon}(\bm{u}) - \bm{\mu^{\epsilon}}\cdot\nabla\Psi\right)\bm{n} = \tilde{B} \quad \text{on\,\,} \Gamma_{mag,N}.
\]
Introducing the function spaces
\begin{align*}
\mathcal{V} &\defeq (H^{1}(\Omega))^{3},   \quad
\mathcal{V}_{0}  \defeq \{ \bm{v} \in \mathcal{V}: v_{i} = 0 \text{\,on\,} \Gamma_{D},\, i = 1,2,3 \}, \\
\mathcal{M} & \defeq H^{1}(\Omega), \qquad
\mathcal{M}_{0}  \defeq \{ \Phi \in \mathcal{M}: \Phi = 0 \text{\,on\,} \Gamma_{m,D} \},
\end{align*}
the corresponding weak form of the coupled problem reads:

\noindent\textbf{Problem $(CP)^{w}$}. Find $(\bm{u},\Psi) \in \mathcal{V}_{0} \times \mathcal{M}_{0}$ such that
\begin{eqnarray}
a(\bm{u,v})+c(\bm{v},\Psi) &= & l(\bm{v}) \qquad\qquad \;\;\, \forall \bm{v}\in \mathcal{V}_{0},  \label{eq:cpa}\\ 
d(\bm{u},\Phi) - b(\Psi,\Phi) &= & m(\Phi) \qquad\qquad \forall \Phi \in \mathcal{M}_{0}.
\label{eq:cpb}
\end{eqnarray} 
Here, 
\begin{eqnarray}
a(\bm{u},\bm{v}) &=& \limint_{\Omega}  \nabla\bm{v}:(\bm{C^{H}}:\bm{\epsilon}(\bm{u})) \: \mathrm{d}x = \limint_{\Omega}(\bm{C^{H}}:\bm{\epsilon}(\bm{u})):\bm{\epsilon}(\bm{v}) \: \mathrm{d}x, \\  \label{eq:bilinearforma}
b(\Psi,\Phi) &=& \limint_{\Omega} \nabla\Phi\cdot(\bm{\mu^{\epsilon}}\cdot\nabla\Psi) \: \mathrm{d} x, \label{eq:bilinearformb}
\end{eqnarray}
\begin{eqnarray}
c(\bm{v},\Psi) &=& \half \limint_{\Omega} \nabla\bm{v}:(\bm{e}\cdot \nabla\Psi) \: \mathrm{d}x = \half \limint_{\Omega} \bm{\epsilon}(\bm{v})^{T} : (\bm{e}\nabla \Psi) \dx \nonumber \\ 
&=&
 \half \limint_{\Omega} \nabla \Psi \cdot (\bm{e}^{T} : \bm{\epsilon}(\bm{v})) \dx, \\ \label{eq:bilinearformc}
l(\bm{v}) &=& \limint_{\Gamma_{N}} \bm{v}\cdot\bm{\tau} \: \mathrm{d}s,  \qquad  m(\Phi) = \limint_{\Gamma_{m,N}} \Phi \tilde{B} \: \mathrm{d}s. \label{eq:linearforms}
\end{eqnarray}
The bilinear forms $a:\mathcal{V}_{0}\times \mathcal{V}_{0}\rightarrow \R$ and $b:\mathcal{M}_{0}\times \mathcal{M}_{0} \rightarrow \R$ correspond to the purely elastic and purely magnetic influence, respectively, the bilinear form $c:\mathcal{V}_{0}\times \mathcal{M}_{0} \rightarrow \R$ describes the coupling between the two fields and the linear functionals $l: \mathcal{V}_{0} \rightarrow \R$ and $m:\mathcal{M}_{0}\rightarrow \R$ represent the magnetic and elastic surface loads. 

Before continuing with the analysis of the coupled problem (\ref{eq:cpa})-(\ref{eq:cpb}),we point out that the minus sign in front of the bilinear form $b$ in (\ref{eq:cpb}) is a consequence of the choice of independent variables in the coupled constitutive equations as well as the resulting energy formulation and strongly influences the properties of the obtained coupled system.

\section{Existence and uniqueness of the solution}

We continue with several important properties of the above bilinear forms that will be used below to prove the existence and uniqueness of the solution to the coupled problem.

\begin{lemma}\label{continuity}
The bilinear forms $a:\mathcal{V}_{0} \times \mathcal{V}_{0} \rightarrow \R$,
$b:\mathcal{M}_{0}\times \mathcal{M}_{0} \rightarrow \R$ and $c:\mathcal{V}_{0} \times \mathcal{M}_{0} \rightarrow \R$ are continuous.
\end{lemma}
\begin{proof}
Since $a$ is the bilinear form of (uncoupled) linear elasticity, its continuity is already known (see, e.g., \cite{Braess2007}). 
To prove the continuity of the bilinear form $b$, we have to show that there exists a constant $C>0$ such that
\[
|b(\Psi,\Phi)| \leq C \|\Psi\|_{1} \|\Phi\|_{1} \qquad \text{for all} \quad \Psi,\Phi \in \mathcal{M}_{0},
\]
where $\|\Psi\|_{1}$
 is the $H^{1}(\Omega)$-norm.
Using \eqref{eq:bilinearformb}, it holds
\begin{eqnarray*}
	b(\Psi,\Phi) &=& \limint_{\Omega} \sum\limits_{i=1}^{3} \mu_{ii}(\frac{\partial\Psi}{\partial x_{i}})(\frac{\partial\Phi}{\partial x_{i}})\dx 
	\leq \underbrace{\max\limits_{i=1,2,3}\{\mu_{ii}\}}_{\eqqcolon \mu_{0}} \limint_{\Omega} |\sum\limits_{i=1}^{3}\frac{\partial \Psi}{\partial x_{i}} \frac{\partial\Phi}{\partial x_{i}}| \dx \\
	&\leq& \mu_{0} \sum\limits_{i=1}^{3}\sqrt{\limint_{\Omega}\left(\frac{\partial \Psi}{\partial x_{i}}\right)^{2}\dx \limint_{\Omega}\left(\frac{\partial \Phi}{\partial x_{i}}\right)^{2} \dx},
	\end{eqnarray*}
	where the last step follows from the Cauchy-Schwarz-inequality  \cite{McCluer2009}.
	Moreover, 
	\begin{eqnarray*}
 \|\Psi\|_{1} \|\Phi\|_{1} &\geq& |\Psi|_{1}|\Phi|_{1} = \sqrt{\limint_{\Omega} \sum_{i=1}^{3}(\frac{\partial \Psi}{\partial x_{i}})^{2} \dx} \sqrt{\limint_{\Omega} \sum_{i=1}^{3}\frac{\partial \Phi}{\partial x_{i}})^{2}\dx}\\
	&=&\sqrt{\sum_{i=1}^{3}\sum_{j=1}^{3}\limint_{\Omega}(\frac{\partial \Psi}{\partial x_{i}})^{2} \dx \limint_{\Omega}(\frac{\partial \Phi}{\partial x_{j}})^{2} \dx}
	\geq  \sqrt{ \sum_{i=1}^{3}\limint_{\Omega}(\frac{\partial \Psi}{\partial x_{i}})^{2} \dx \limint_{\Omega}(\frac{\partial \Phi}{\partial x_{i }})^{2} \dx.}
	\end{eqnarray*}
	It holds $
	\sqrt{\half A + \half B} > \half \sqrt{A} + \half \sqrt{B} \, \forall A, B 
	\in \R^+, A\neq B $, and thus
\begin{eqnarray*}
&\underbrace{\mu_{0} \frac{4}{\sqrt{2}}}_{\eqqcolon C}&\|\Psi\|_{1}\|\Phi\|_{1}
 \geq \mu_{0} \frac{4}{\sqrt{2}}\sqrt{ \sum_{i=1}^{3}\limint_{\Omega}(\frac{\partial \Psi}{\partial x_{i}})^{2} \dx \limint_{\Omega}(\frac{\partial \Phi}{\partial x_{i}})^{2} \dx}\\
&\geq& 2\mu_{0} \left(\sqrt{\limint_{\Omega}\left(\frac{\partial \Psi}{\partial x_{1}}\right)^{2} \dx \limint_{\Omega}\left(\frac{\partial \Phi}{\partial x_{1}}\right)^{2} \dx}
 +  \sum_{i=2}^{3}\sqrt{\limint_{\Omega}\left(\frac{\partial \Psi}{\partial x_{i}}\right)^{2} \dx \limint_{\Omega}\left(\frac{\partial \Phi}{\partial x_{i}}\right)^{2} \dx } \right)\\
  &\geq& b(\Psi,\Phi).
\end{eqnarray*}
The continuity of $c$ can be proven in a similar manner.
\end{proof}

Since the bilinear forms $a$ and $b$ are obviously symmetric and non-negative for all $\bm{v}\in \mathcal{V}_{0}$ and $\Phi \in \mathcal{M}_{0}$, the coupled system can be interpreted as a saddle point problem with penalty term $- b(\Psi,\Phi)$, see, e.g., \cite{Braess2007}. To address the question about the existence and uniqueness of solutions to such 
penalized saddle point problems, we have to consider the solvability of classical saddle point problems in the first instance. The corresponding criteria are given by \textit{Brezzi's Splitting Theorem} \cite{Brezzi1974}, which states that a saddle point problem in the structure of $(CP)^{w}$ without the term $- b(\Psi,\Phi)$ and with continuous bilinear forms $a$ and $c$, where, additionally, $a$ is symmetric and non-negative on $\mathcal{V}_{0}$,  has a unique solution, if and only if the following conditions hold:
 \begin{enumerate}[(i)]
\item  The bilinear form $a$ is elliptic on $\ker\mathcal{C} \defeq \{\bm{v}\in \mathcal{V}_{0}|\quad c(v, \Psi) = 0 \quad \forall  \Psi\in \mathcal{M}_{0}\}$, i.e.
\[ \exists \alpha > 0 \quad \text{s.t.}\quad a(\bm{v},\bm{v}) \geq \alpha \norm{\bm{v}}_{\mathcal{V}_{0}}^{2} \quad \forall \bm{v} \in \mathcal{V}_{0}. \]
\item The bilinear form $c$ satisfies the inf-sup-condition or Lady\v{z}enskaya-Babu\v{s}ka-Brezzi-condition
\[ \exists \beta > 0 \quad \text{s.t.}\quad \inf\limits_{\Psi\in \mathcal{M}_{0}}\sup\limits_{\bm{v} \in \mathcal{V}_{0}} \frac{c(v, \Psi)}{\norm{\bm{v}}_{\mathcal{V}_{0}}\norm{\Psi}}_{\mathcal{M}_{0}}\geq \beta. \]
\end{enumerate}

For penalized saddle point problems, additional conditions have to be imposed on the bilinear forms. A sufficient condition is the boundedness and non-negativity of the bilinear form defining the penalty term (for a proof, refer to \cite{Braess2007, Braess1996}). In this case, the solution is even uniformly bounded (see \cite{Braess2007}, Chapter 3). The following theorem shows that all conditions stated above are satisfied for the saddle point problem $(CP)^{w}$\n\n\n. Before stating the theorem, we need to characterize a certain property of the coupling matrix $\bm{e}$ in the first instance.
\begin{definition}(Minimal positivity)\label{def:minimalpositivity}
Let $\bm{A}$ be a matrix in $\R^{6 \times 3}\nn\n.$  \,\,\,\,We say that $\bm{A}$ satisfies the \textit{minimal positivity} property iff its components $a_{ij}$ fulfill the requirement
\begin{equation}\label{eq:minimalpositivity}
\min_{k = 1,...,6} \hat{A}_{k} > 0,
\end{equation}
with $\hat{A}_{k}$ defined as
\begin{align*}
\hat{A}_{1} &\defeq  \half (a_{11}+ a_{51} + a_{61}), \quad \hat{A}_{2} \defeq  \half (a_{22}+ a_{42} + a_{62}) \\
\hat{A}_{3} &\defeq  \half (a_{33}+ a_{43} + a_{53}), \quad \hat{A}_{4} \defeq  \half (a_{12}+a_{21}+a_{41}+a_{52}+a_{61}+a_{62}) \\
\hat{A}_{5} &\defeq  \half (a_{13}+a_{31}+a_{41}+a_{51}+a_{53}+a_{63}), \quad \hat{A}_{6} \defeq  \half (a_{23} + a_{32}+a_{42}+a_{43}+a_{52}+a_{63}).
\end{align*}
 \end{definition}
\begin{theorem}\label{thm:existenceuniqueness}
The saddle point problem $(CP)^{w}$ with (bi-)linear forms as defined in \eqref{eq:bilinearforma}-\eqref{eq:linearforms} has a unique and uniformly bounded solution if the coupling matrix $\bm{e}$ fulfills the minimal positivity requirement \eqref{eq:minimalpositivity}.
\end{theorem}
\begin{proof}
Lemma \ref{continuity} states that all three bilinear forms are continuous on the corresponding spaces. Moreover, both $a(\cdot,\cdot)$ and $b(\cdot,\cdot)$ are obviously symmetric and non-negative on $\mathcal{V}_{0}$ and $\mathcal{M}_{0}$, respectively. The coercivity of $a(\cdot,\cdot)$ follows directly from Korn's inequality (a proof can be found e.g. in \cite{Braess2007} or \cite{Ciarlet1978}). In fact, $a$ is coercive on the whole space $\mathcal{V}_{0}$ rather than just on $\ker \mathcal{C}\subset \mathcal{V}_{0}$. To show the inf-sup condition for the bilinear form $c(\bm{v},\Psi)$, we will prove the equivalent formulation: Show that $\exists \beta>0$ s.t.
\begin{equation*}
\sup\limits_{\bm{v}\in \mathcal{V}_{0}} \frac{c(\bm{v},\Psi)}{\|\bm{v}\|_{\mathcal{V}}} \geq \beta \|\Psi\|_{\mathcal{M}} \quad \forall \Psi\in \mathcal{M}_{0}.
\end{equation*}
The coupled bilinear form $c(\bm{v},\Psi)$ is given by  
\begin{eqnarray*}
c(\bm{v},\Psi) &=& \half\limint_{\Omega} \sum\limits_{j = 1}^3\frac{\partial v_{j}}{\partial x_{j}}\left(\sum_{i=1}^{3}e_{ji}\frac{\partial \Psi}{\partial x_{i}}\right)
+ \left(\frac{\partial v_{2}}{\partial x_{3}} + \frac{\partial v_{3}}{\partial x_{2}}\right)\left(\sum_{i=1}^{3}e_{4i}\frac{\partial \Psi}{\partial x_{i}}\right)\\
&+& \left(\frac{\partial v_{1}}{\partial x_{3}} + \frac{\partial v_{3}}{\partial x_{1}}\right) \left(\sum_{i=1}^{3}e_{5i}\frac{\partial \Psi}{\partial x_{i}}\right) 
+ \left(\frac{\partial v_{1}}{\partial x_{2}} + \frac{\partial v_{2}}{\partial x_{1}}\right)\left(\sum_{i=1}^{3}e_{6i}\frac{\partial \Psi}{\partial x_{i}}\right) \dx. 
\end{eqnarray*}
Choosing $v_{i} = \Psi$ for $i = 1,2,3$, the bilinear form transforms into
\begin{equation*}
c(\bm{v},\Psi) = c^{*}(\Psi,\Psi)\defeq \half\limint_{\Omega} 2\sum_{i = 1}^3 \left(\frac{\partial\Psi}{\partial x_{i}}\right)^2 \hat{E}_{i} + \frac{\partial\Psi}{\partial x_{1}}\frac{\partial\Psi}{\partial x_{2}}\hat{E}_{4} + \frac{\partial \Psi}{\partial x_{1}}\frac{\partial\Psi}{\partial x_{3}}\hat{E}_{5} + \frac{\partial \Psi}{\partial x_{2}}\frac{\partial\Psi}{\partial x_{3}}\hat{E}_{6}\dx
\end{equation*}
with $\hat{E}_{i}$, $i=1,..,6$ as given in Definition~\ref{def:minimalpositivity}.\ Since $\bm{e}$ satisfies the minimal positivity property \eqref{eq:minimalpositivity}, i.e.\ $M\defeq\min\limits_{i = 1,...,6}\hat{E}_{i} > 0,$ we obtain 
\begin{eqnarray*}
c(\bm{v},\Psi) &\geq& \frac{1}{4}M\limint_{\Omega} 4 \sum_{i = 1}^3 \left(\frac{\partial\Psi}{\partial x_{i}}\right)^2  + \sum_{\substack{i,j = 1\\ i\neq j}}^3 \frac{\partial\Psi}{\partial x_{i}}\frac{\partial\Psi}{\partial x_{j}} \dx
= \frac{1}{4}M\limint_{\Omega} \sum_{i = 1}^3 \left(\frac{\partial\Psi}{\partial x_{i}}\right)^2 \dx \\
&+& \frac{1}{4}M\limint_{\Omega} \sum_{i = 1}^3 \left(\frac{\partial\Psi}{\partial x_{i}}\right)^2 + \left(\frac{\partial\Psi}{\partial x_{1}}+\frac{\partial\Psi}{\partial x_{2}}\right)^2 + \left(\frac{\partial\Psi}{\partial x_{1}}+\frac{\partial\Psi}{\partial x_{3}}\right)^2 + \left(\frac{\partial\Psi}{\partial x_{2}}+\frac{\partial\Psi}{\partial x_{3}}\right)^2 \dx.
\end{eqnarray*}
Hence,
\begin{equation*}
c(\bm{v},\Psi) \geq  \frac{1}{4}M|\bm{v}|_{\mathcal{V}} + \frac{1}{4}M\limint_{\Omega} \bm{\epsilon}(\bm{v})\cdot\bm{\epsilon}(\bm{v}) \dx
\geq\frac{1}{4}M C'\|\bm{v}\|_{\mathcal{V}}^{2},
\end{equation*}
where $|\bm{v}|_{\mathcal{V}}$ is the seminorm in $\mathcal{V}=H^1(\Omega)^3$ and the constant
$C' = C'(\Omega, \Gamma_{D})$ stems from Korn's inequality. Furthermore, we have
\begin{equation*}
\|\bm{v}\|_{\mathcal{V}} = \sqrt{\sum_{i=1}^{3}(\|v_{i}\|_{0} +|v_{i}|_{1})^{2}} \geq \sqrt{\sum_{i=1}^{3} |v_{i}|_{1}^{2}}
= \sqrt{\limint_{\Omega} \sum_{i,j=1}^{3}\left(\frac{\partial v_{i}}{\partial x_{j}}\right)^{2}\dx},
\end{equation*}
where $\norm{\cdot}_{0}$ and  $\norm{\cdot}_{1}$ denote the $L^{2}(\Omega)$-norm and the $H^{1}(\Omega)$-norm, respectively. 
On the other hand, 
\begin{equation*}
\|\Psi\|_{1} =\sqrt{(\|\Psi\|_{0} + |\Psi|_{1})^{2}} 
		\leq \sqrt{s^{2}+1}|\Psi|_{1},
\end{equation*}		
with a constant $s=s(\Omega)$ from the Poincaré-Friedrichs inequality \cite{Braess2007}, assuming that $\Omega$ can be included in a cube with side length $s$. Using the above estimates for $\|\bm{v}\|_{\mathcal{V}}$ and $\|\Psi\|_{1}$ and setting $\beta \defeq \frac{\sqrt{3}MC'}{4\sqrt{s^2+1}},$ we derive the following chain of inequalities: 
\begin{equation*}
\sup\limits_{\bm{v}\in \mathcal{V}_{0}} \frac{c(\bm{v},\Psi)}{{\norm{\bm{v}}}_{\mathcal{V}}}
 \geq\frac{c(\bm{v},\Psi)}{{\norm{\bm{v}}}_{\mathcal{V}}}
 \geq \frac{1}{4}M C'\norm{\bm{v}}_{\mathcal{V}} \\
\geq \frac{\sqrt{3}}{4}M C'\sqrt{\limint_{\Omega}\sum_{i = 1}^3 \left(\frac{\partial\Psi}{\partial x_{i}}\right)^2 \dx} \geq \beta\norm{\Psi}_{1}.
\end{equation*}
\end{proof} 
\begin{remark}
Since the constant in the inf-sup estimate depends on the components of the coupling matrix $\bm{e}$, for weakly magnetostrictive materials, i.e.\ for $e_{ij}\rightarrow 0$, the estimate in the proof of Theorem~\ref{thm:existenceuniqueness} does not provide a satisfactory bound. However, most of the magnetostricive materials used in industrial applications are so-called \textit{Giant Magnetostrictive Materials} for which the quantity $M$ from the proof of Theorem~\ref{thm:existenceuniqueness} can be assumed to be ``sufficiently large".
\end{remark}

In fact, the $\mathcal{V}_{0}$-ellipticity - condition is also fulfilled for the bilinear form $b$:
\begin{lemma}\label{coercivity}
The bilinear form $b$ defined in \eqref{eq:bilinearformb} is coercive on the space $\mathcal{M}_{0}$.
\end{lemma}
\begin{proof}
For $\Psi\in \mathcal{V}_{0}$, the following estimates hold:
\begin{eqnarray*}
r \norm{\Psi}_{1}^2 &=& r (\norm{\Psi}_{0}^2 + |\Psi|_{1}^{2}) \leq r (s^{2}+1) |\Psi|_{1}^{2} 
= r (s^2+1) \limint_{\Omega} \sum_{i=1}^{3} \left(\frac{\partial\Psi}{\partial x_{i}}\right)^{2}\dx\\
&\leq& (s^2+1) \limint_{\Omega} \sum_{i=1}^{3} \mu_{ii}\left(\frac{\partial\Psi}{\partial x_{i}}\right)^{2}\dx = (s^2+1) b(\Psi,\Psi),
\end{eqnarray*}
where the parameter $s$ stems from the Poincar\'{e}-Friedrichs-inequality and \\$r \defeq \min\limits_{i=1,2,3} \mu_{ii}$. Thus, coercivity holds with the constant $\gamma \defeq \frac{r}{(s^{2}+1)}$.
\end{proof}
 As a consequence, both $a$ and $b$ are coercive and define inner products on the corresponding spaces. Moreover, the coercivity of $b$ can be used to show the unique solvability of the problem $(CP)^{w})$ directly from the application of the Lax-Milgram-Lemma (following \cite{BoffiBrezziFortin2013}) to the combined bilinear form 
\[
\mathcal{A}:(\mathcal{V}_{0}\times \mathcal{M_{0}})\times (\mathcal{V}_{0}\times \mathcal{M_{0}}) \rightarrow \R, \quad e((\bm{u},\Psi), (\bm{v},\Phi)) = a(\bm{u},\bm{v}) + c(\bm{v},\Psi) - c(\bm{u},\Phi) + b(\Psi,\Phi),
\]
which transfers the system of two coupled equations to a problem with a single equation: \\
Find $(\bm{u},\Psi) \in \mathcal{V}_{0} \times \mathcal{M}_{0}$, such that
\begin{equation*}
\mathcal{A}((\bm{u},\Psi), (\bm{v},\Phi)) = l(\bm{v}) - m(\Phi) \quad \forall (\bm{v},\Phi) \in \mathcal{V}_{0}\times M_{0}.
\end{equation*} 
Using the coercivity of the bilinear forms $a$ and $b$ (as well as the continuity of all three bilinear forms), we can estimate 
\begin{eqnarray*}
\alpha\norm{\bm{v}}_{\mathcal{V}_{0}}^{2} + \gamma \norm{\Phi}_{\mathcal{M}_{0}}^{2} &\leq& a(\bm{v},\bm{v}) + b(\Psi,\Phi) \\
&=& a(\bm{v},\bm{v}) + b(\Phi,\Phi) + c(\bm{v},\Phi) - c(\bm{v},\Phi) = \mathcal{A}((\bm{v},\Phi), (\bm{v},\Phi)), 
\end{eqnarray*}
implying that there exists a $\delta \defeq \min\{\alpha,\gamma\}$ such that $e((\bm{v},\Phi), (\bm{v},\Phi))\geq \delta\norm{(\bm{v},\Phi)}_{\mathcal{V}_{0}\times\mathcal{M}_{0}}$ for all $(\bm{v},\Phi)\in\mathcal{V}_{0}\times\mathcal{M}_{0}$, where $\norm{(\bm{v},\Phi)}_{\mathcal{V}_{0}\times\mathcal{M}_{0}}^2 \defeq \norm{\bm{v}}_{\mathcal{V}_{0}}+\norm{\Phi}_{\mathcal{M}_{0}}$. Hence, the combined bilinear form $e$ is continuous and coercive and Lax-Milgram yields the existence and uniqueness of the solution $(\bm{u},\Psi)$ of the above problem.\\
However, exploiting the coercivity of the uncoupled bilinear forms does not necessarily let us avoid showing the inf-sup-condition for $c$.  In fact, the stability estimate that can be obtained by using just the coercivity properties of $a$ and $b$  does not yield a satisfactory result for our problem. A detailed discussion of the stability aspects of penalized saddle point problems can be found in \cite{BoffiBrezziFortin2013}.

\section{Solvability in the discrete case}\label{subsec:discretesolvability}

Consider now finite dimensional subspaces $V^{h}\subset \mathcal{V}_{0}$ and $M^{h}\subset \mathcal{M}_{0}$ of the function spaces defined above, where $h > 0$ denotes the mesh parameter. Then, the Galerkin approximation of the coupled saddle point problem reads

\noindent\textbf{Problem $(CP)^{w}_{h}$}. Find $(\bm{u}^{h},\Psi^{h}) \in V^{h} \times M^{h}$, such that
\begin{eqnarray*}
a(\bm{u}^{h},\bm{v}^{h})+c(\bm{v}^{h},\Psi^{h}) &=& l(\bm{v}^{h}) \qquad\qquad \;\;\, \bm{v}^{h}\in V^{h} \\
c(\bm{u}^{h},\Phi^{h}) - t^{2} b(\Psi^{h},\Phi^{h}) &=& m(\Phi^{h}) \qquad\qquad \Phi^{h} \in M^{h}.
\end{eqnarray*}
In contrast to problems with a single bounded, coercive bilinear form, that can be studied using the Lax-Milgram theorem (e.g. in~\cite{Braess2007} or \cite{Brennerscott1994}), the well-posedness of discrete saddle point problems, in general, cannot be automatically deduced from their continuous counterpart:  On the one hand, the coercivity of the bilinear form $a$ on $\ker(\mathcal{C})$ does not necessarily imply its coercivity on $(\ker\mathcal{C})^{h}\defeq \{\bm{v}^{h}\in V^{h}| c(\bm{v}^{h}, \Psi^h) = 0 \,\,\forall \Psi^{h}\in M^{h}\}$ (which is the case, for example, for the Stokes Problem (e.g. in ~\cite{Braess2007}). On the other hand, the discrete inf-sup condition on $V^{h} \times M^{h}$ cannot be directly deduced from the continuous one on $\mathcal{V}_{0}\times \mathcal{M}_{0}$ since the supremum over a set $Y$ is greater than or equal to the supremum over a subset $Y^{h}\subset Y$.\\
In case of the penalized saddle point problem $(CP)^w$, however, the continuity of the bilinear form $c$ as well as the coercivity of the bilinear form $a$ on the corresponding discrete spaces are directly inherited from the continuous case. The discrete inf-sup condition is then the only assumption that needs verification. In this context, special attention has to be paid to the appropriate choice of the discrete spaces $V^{h}$ and $M^{h}$, since it is crucial in satisfying the assumption. 

If we select Lagrangian finite elements for the discretization, however, we can show that the discrete inf-sup condition is satisfied. To this end, the same mesh ${\cal T}_h$ and the same piecewise polynomial approximation are used for both subspaces $\mathcal{V}^h$ and 
$\mathcal{M}^h$ by setting $\mathcal{V}^h = S_k(\Omega,{\cal T}_h)^3$ and  
$\mathcal{M}^h = S_k(\Omega,{\cal T}_h)$ where $S_k(\Omega,{\cal T}_h)$ stands for the 
conforming Lagrangian finite element space of degree $k$.
Then the bilinear form $c$ satisfies the discrete inf-sup condition with a constant that is independent of the mesh parameter $h$, which yields an optimal convergence rate of the Galerkin method for the coupled problem:

\begin{theorem}\label{th:discreteinfsup}
For $\mathcal{V}^h = S_k(\Omega,{\cal T}_h)^3$ and  
$\mathcal{M}^h = S_k(\Omega,{\cal T}_h)$,
 the bilinear form $c$ defined in \eqref{eq:bilinearformc} satisfies the discrete inf-sup condition
\begin{equation}\label{eq:discreteinfsup}
\exists \hat{\beta}>0, \quad \text{such that} \quad \inf\limits_{\Psi^{h}\in \mathcal{M}^{h}}\sup\limits_{\bm{v}^{h}\in \mathcal{V}^{h}}\frac{c(\bm{v}^{h},\Psi^{h})}{\|\bm{v}^{h}\|_{\mathcal{V}^{h}}\|\Psi_{h}\|_{\mathcal{M}^{h}}}\geq \hat{\beta}, 
\end{equation}
 with a constant $\hat{\beta}$ independent of $h$.
\end{theorem}
The proof proceeds along the same lines as the proof in the continuous case.
The important feature is the special choice of the elastic and magnetic basis functions as polynomials of the same degree, which enables the use of Korn's inequality. 
 
 \section{Conclusion}
In this paper, we derived a coupled magneto-elasto-static model under the assumption of linear and reversible material behavior using a minimum energy principle. As a result, we obtained a system of coupled partial differential equations for the elastic displacement and the magnetic scalar potential. In its weak form, the system contains three different bilinear forms of elastic, magnetic and coupled quantities and represents a saddle point problem, for which the existence, uniqueness and uniform boundedness of the solution were proven for the continuous and discrete cases. In particular, we have shown that, in both cases, the coupled bilinear form satisfies an inf-sup condition which is essential for the stability of the problem.

\footnotesize{\bibliography{Paper_Harutyunyan}}

\end{document}